\pdfoutput=1 
\documentclass[a4paper,12pt,reqno]{amsart} 
\usepackage[margin=1.25in]{geometry}
\usepackage{amssymb} 
\usepackage{enumerate} 
\usepackage{enumitem} 
\usepackage[parfill]{parskip} 
\usepackage{hyperref} 
\usepackage{graphicx}
\usepackage{tikz-cd}
\usepackage{float}
\usepackage{xpatch}
\allowdisplaybreaks[2]
\newtheoremstyle{test}{4pt}{4pt}{\itshape}{}{\bfseries}{.}{0.3cm}{} \theoremstyle{test} 
\newtheorem{thm}{Theorem}[subsection] 
\newtheorem{lem}[thm]{Lemma} 
\newtheorem{prop}[thm]{Proposition}

\newtheorem{thmx}{Theorem}

\newtheorem*{thmx-non}{Theorem}
\theoremstyle{definition} 
\newtheorem{defn}[thm]{Definition} 
 
\theoremstyle{remark} 
\newtheorem*{rem}{Remark} 
 
\newlist{steps}{enumerate}{1} 
\setlist[steps]{wide=0pt,leftmargin=\parindent, label=Step \arabic*:,font=\bfseries} 
\newcommand{\eps}{\varepsilon}

\makeatletter
\xpatchcmd{\proof}{\topsep6\p@\@plus6\p@\relax}{}{}{}
\makeatother

\makeatletter
\newcommand{\inlinetag}[1]{%
  \def\@currentlabel{*}%
  \label{#1}%
  \quad\tagform@{*}%
}
\makeatother

\begin{document} 
\sloppy 

\title{A Quantitative Guessing Geodesics Theorem}
\author[T. Shlomovich]{Talia Shlomovich}
\address{Department of Mathematics, Heriot-Watt University and Maxwell Institute for Mathematical Sciences, Edinburgh, UK}
\email{ls3007@hw.ac.uk}

\begin{abstract}
    We present a quantitative version of the guessing geodesics criterion, which is a well-known theorem that provides a set of conditions to prove hyperbolicity of a given metric space. This version adds to the existing result by determining an explicit estimate of the hyperbolicity constant. As a sample application of this result, we estimate the hyperbolicity constant for a particular hyperbolic model of \(\mathrm{CAT}(0)\) spaces known as the curtain model.

    
\end{abstract}

\maketitle

\section{Introduction}
One of the most broadly applicable ways of showing that a given metric space is Gromov hyperbolic is the \emph{Guessing Geodesics} theorem, which has appeared in various forms. 
This theorem has been presented in several ways: Bowditch \cite[Proposition~3.1]{B}, Masur-Schleimer \cite[Theorem~3.15]{MS}, and Hamenst\"{a}dt \cite[Proposition~3.5]{UH}. The latter is closest to the version we use in this paper, and is stated below as follows.
\begin{thmx}[Guessing Geodesics, \cite{UH}]
    Let \((X,d)\) be a geodesic metric space. Assume that for some constant \(D > 0\) there are paths \(\eta_{xy} = \eta(x, y) : [0, 1] \to X\) from \(x\) to \(y\), for each pair \(x,y \in X\). Suppose also that the following conditions are satisfied:
\begin{enumerate}[label=(G\arabic*)]
    \item If \(d(x,y) \leq 1\), then the diameter \(\mathrm{diam}(\eta_{xy}) \leq D\).
    \item For any \(s \leq t\), the Hausdorff distance between the paths \(\eta_{xy}[s, t]\)
    and \(\eta(\eta_{xy}(s), \eta_{xy}(t))\) is at most \(D\).
    \item For any \(x, y, z \in X\), the set \(\eta_{xy}\) is contained in the \(D\)-neighbourhood
    of \(\eta_{xz} \cup \eta_{zy}\).
\end{enumerate}
Then \((X,d)\) is \(\delta\)-hyperbolic for some \(\delta = \delta(D) > 0\).
\end{thmx}

Roughly, the theorem states that in order to show hyperbolicity, it suffices to find a system of paths that ``look'' like geodesics in a hyperbolic space. 
This idea has been used in various places: Bowditch and Hamenst\"adt used it to show that curve complexes are hyperbolic \cite{B,UH}, Dru\c{t}u--Sapir to show that relatively hyperbolic groups are asymptotically tree-graded \cite[Section~9]{DS} and Sisto to characterise hyperbolically embedded subgroups \cite{S, DGO}.
However, the existing statements of the guessing geodesics theorem do not give explicit hyperbolicity constants, though in Hensel--Przytycki--Webb \cite{HPW}, guessing geodesics was used to prove that the curve graph is \(17\)-hyperbolic.

Recall that a (not necessarily geodesic) metric space \(X\) is \({\delta}\)\emph{--hyperbolic} in the sense of Gromov if for all \(x,y,z,w \in X\),
        \[
        d(x,z) + d(y,w) \leq \max\{d(x,w)+ d(y,z), d(x,y) + d(z,w)\} + 2\delta.
        \]

The first quantitative guessing geodesics result provides a set of conditions to prove hyperbolicity with an explicit hyperbolicity constant in geodesic metric spaces. To avoid confusion, throughout this paper we use \(\delta\) to refer to a hyperbolicity constant, and \(\delta'\) to refer to a quasi-hyperbolicity constant.
\begin{thmx}[Quantitative Guessing Geodesics]
\label{thmx:a}
    Let \((X,d)\) be a geodesic metric space. Assume that for some constant \(D > 0\) there are \(D\)-coarsely connected paths \(\eta_{xy}  = \eta(x,y) : [0,1] \to X\) from \(x\) to \(y\), for each pair \(x,y \in X\). Suppose \(X\) also satisfies the following conditions:
    \begin{enumerate}[label=(G\arabic*)]
        \item The diameter \(\mathrm{diam}(\eta_{xy}) \leq \frac{D d(x,y)}{2}\),
        \item For any \(s \leq t\), \(d_{\mathrm{Haus}}(\eta_{xy}[s,t], \eta(\eta_{xy}(s),\eta_{xy}(t)) \leq D\),
        \item For any \(x,y,z \in X\), \(\eta_{xy} \subset \mathcal{N}_D(\eta_{xz} \cup \eta_{zy})\).
    \end{enumerate}

Also define \(f(x) := D\log_2(8x+2D)+ \frac{3}{2}D\). Then every geodesic triangle in \(X\) is \(\delta\)--thin, where \(z\) is any number such that \(f(z) < z\) and \(\delta = 4f(z)+2D\).
\end{thmx}

However, one might be interested in an explicit hyperbolicity estimate for non-geodesic hyperbolic spaces. One such example is the free factor complex of a free group \(F_n\). This motivates a version of the theorem extending beyond geodesic spaces, which forms the main focus of this paper.
\begin{thmx}[Rough Guessing Geodesics]
    \label{thmx:mainresult}
    Let \((X,d)\) be a \((1,q)\)--quasi-geodesic metric space.
    Assume that for some constant \(D > 0\) there are \(D\)-coarsely connected paths \(\eta_{xy}  = \eta(x,y) : [0,1] \to X\) from \(x\) to \(y\), for each pair \(x,y \in X\). Suppose \(X\) also satisfies the following conditions:
    \begin{enumerate}[label=(G\arabic*)]
        \item The diameter \(\mathrm{diam}(\eta_{xy}) \leq \frac{D d(x,y)}{2}\),
        \item For any \(s \leq t\), \(d_{\mathrm{Haus}}(\eta_{xy}[s,t], \eta(\eta_{xy}(s),\eta_{xy}(t)) \leq D\),
        \item For any \(x,y,z \in X\), \(\eta_{xy} \subset \mathcal{N}_D(\eta_{xz} \cup \eta_{zy})\).
    \end{enumerate}
    
    Then every \((1,q)\)--quasi-geodesic triangle in \(X\) is \(\delta'\)-thin, and \((X,d)\) is \(\delta\)--hyperbolic, where \(\delta' \leq 72qD^{\frac{5}{4}} + 2D + 3q\) and \(\delta = 56\delta' + 6q\).
\end{thmx}

Our most general guessing geodesics result applying to quasi-geodesic metric spaces is given by Proposition~\ref{prop:guessinggeodesics}. Here the quasi-hyperbolicity constant \(\delta'\) is expressed in terms of an auxiliary quantity \(\kappa\), and we provide general estimates for the value of \(\kappa\) in Lemma~\ref{lem:kappa1}. In fact, in Lemma~\ref{lem:lambert}, we see that \(\kappa\) can be evaluated explicitly using the Lambert \(W\) function. The statement of Theorem~\ref{thmx:mainresult} above was obtained from Proposition~\ref{prop:guessinggeodesics} by applying Lemma~\ref{lem:kappa1} with \(n=8\). 

As a sample application of our methods, we determine an estimate for the hyperbolicity constant of the \emph{curtain model}, which is a non-geodesic hyperbolic space introduced in \cite{PSZ}. We show that for any \(\mathrm{CAT}(0)\) space \(X\), the curtain model of \(X\) is \((1.19\times10^7)\)--hyperbolic. 

Whilst we provide a general bound for these hyperbolicity constants, in specific contexts it may be more optimal to use one of the intermediate results from Section~\ref{section:guessinggeodesics} and determine an estimate directly. 
By doing so, one may obtain a significantly improved estimate of the hyperbolicity constant, which is \((5.56 \times 10^5)\)--hyperbolic.

\section*{Acknowledgements}
I am truly grateful to Harry Petyt for his invaluable mentorship and guidance throughout this project, without which this paper would not have been possible. I also thank Davide Spriano for many helpful discussions and for his generosity in offering assistance at short notice. My thanks also go to the anonymous reviewer for a careful reading of the manuscript and several helpful suggestions, including the explicit computation of \(\kappa\) using the Lambert \(W\) function.

This research was supported by the UKRI Centre for Doctoral Training in Algebra, Geometry and Quantum Fields (AGQ), grant number EP/Y035232/1.

\section{Guessing Geodesics}
\subsection{Quantitative guessing geodesics}
\label{section:guessinggeodesics}

In this section, we present an effectivisation of the well-known guessing geodesics criterion in the setting of non-geodesic hyperbolic spaces. We start this section by introducing some relevant terminology.
\begin{defn}
    A metric space \(X\) is called \({c}\)\emph{--coarsely connected} if for every two points \(x,y \in X\), there exists a sequence \(x=x_0,x_1, \dots, x_n=y\) of points in \(X\) such that \(d(x_{i-1},x_i) \leq c\).
\end{defn}

Next we introduce a definition of a useful type of non-geodesic hyperbolic space called a quasi-hyperbolic metric space \cite[\S7.2.2]{Loh}. To do this, we use the thin triangles definition of hyperbolicity to define a similar notion of quasi-hyperbolicity. To avoid confusion between thin and slim triangles, we use the definition of thin triangles below.
\begin{defn}   
    Let \(X\) be a metric space. A \((1,q)\)--quasi-geodesic triangle \(T \subset X\) with sides \(\lambda_1, \lambda_2, \lambda_3\) is \({\delta'}\)\emph{–thin} if \(\lambda_i \subset \mathcal{N}_{\delta'}(\lambda_j \cup \lambda_k)\) for all permutations \((i,j,k)\) of the indices \(\{1,2,3\}\).
\end{defn}
\begin{defn}
    \label{defn:quasihyperbolic1}
    Let \(X\) be a metric space. Suppose there is a constant \(k>0\) such that \(X\) is a \((k,k)\)--quasi-geodesic space, and for every \(q>0\), there is a constant \(\delta'= \delta'(k,q) >0\) such that each \((q,q)\)--quasi-geodesic triangle is \(\delta'\)-thin. We then say that the metric space \(X\) is \({\delta'}\)\emph{--quasi-hyperbolic}.
\end{defn}
\emph{Notation:} For a path \(\lambda: [0,m] \to X\), we write \(\ell(\lambda) = m\) for the parametrisation length of the path \(\lambda\). 

We now present the main result of this paper in its most general form, which establishes a general quasi-hyperbolicity constant. The proof is adapted from \cite[Proposition~A.1]{PSZ}.
\begin{prop} 
\label{prop:guessinggeodesics}
    Let \(X\) be a \((q_1,q_2)\)--quasi-geodesic metric space for \(q_1,q_2>0\).
    Assume that for some constant \(D > 0\) there are \(D\)-coarsely connected paths \(\eta_{xy}  = \eta(x,y) : [0,1] \to X\) from \(x\) to \(y\), for each pair \(x,y \in X\). Suppose \(X\) also satisfies the following conditions:
    \begin{enumerate}[label=(G\arabic*)]
        \item The diameter \(\mathrm{diam}(\eta_{xy}) \leq \frac{D d(x,y)}{2}\),
        \item For any \(s \leq t\), \(d_{\mathrm{Haus}}(\eta_{xy}[s,t], \eta(\eta_{xy}(s),\eta_{xy}(t)) \leq D\),
        \item For any \(x,y,z \in X\), \(\eta_{xy} \subset \mathcal{N}_D(\eta_{xz} \cup \eta_{zy})\).
    \end{enumerate}

    Also define \(f(x) := D\log_2\left(8xq_1^3 + 7q_1^3q_2 + 2Dq_1\right) +\frac{1}{2}(q_1+q_2)D + D\).
    
    Then every \((q_1,q_2)\)--quasi-geodesic triangle in \(X\) is \(\delta'\)--thin, where
    \begin{equation*}
    \delta'(q_1,q_2,D) = 2\left(\frac{q_1^2}{2}(2\kappa + D + q_2) + q_2 + \kappa\right) + D
    \end{equation*}
    and \(\kappa =  f(z)\) for any choice of \(z \geq 0\) for which \(f(z) < z\).
\end{prop}
\begin{rem}
    Some useful applications of this general theorem are:
    \begin{itemize}
    \item If \(X\) is in fact a geodesic metric space, substituting \((q_1,q_2) = (1,0)\) gives Theorem~\ref{thmx:a}. 
    \item If \(X\) is a \((1,q)\)--quasi-geodesic space, the constant \(\kappa\) is determined by \(f(x) = D\log_2(8x+7q+2D)+\frac{1}{2}\left(1+q\right)D + D\), and \(\delta' = 4\kappa + 2D + 3q\).
    \end{itemize}
\end{rem}
\begin{proof}
    For a outline of the proof, we bound the Hausdorff distance between \(\eta_{xy}\) and an arbitrary \((q_1,q_2)\)--quasi-geodesic \(\gamma\) from \(x\) to \(y\) by a constant \(g_{q_1,q_2,D}\), which we express in terms of another constant
    \(\kappa = \kappa(q_1,q_2,D)\). The value of this second constant \(\kappa\) is determined in Section~\ref{section:kappa}.
    After bounding this Hausdorff distance, (G3) implies that \((q_1,q_2)\)--quasi-geodesic triangles are \(\delta' = 2g_{q_1,q_2,D} + D\) thin.

    This proof is lengthy, so we have divided it into five parts for easier comprehension, as outlined below.
    
    \emph{Part 1:}
    Let \(\lambda : [0, 2^n] \to X\) be a \((q_1,q_2)\)-coarsely Lipschitz path between any pair of endpoints in \(X\) (not necessarily from \(x\) to \(y\)). We determine an upper bound for the constant \(c > 0\) such that \(\eta(\lambda(0), \lambda(2^n))\) is in the \(c\)-neighbourhood of \(\lambda\). To do this, we consider a process of iterative subdivision in which we compare the distance between the path \(\eta(\lambda(0), \lambda(2^n))\) and a concatenation of paths which we describe below. 
    
    The \(k\)\textsuperscript{th} step of the subdivision (where \(1 \leq k \leq n\)) involves splitting the interval \([0, 2^n]\) into \(2^k\) subintervals (Fig.~\ref{fig:figure1}). These subintervals give rise to \(2^k\) paths which will be concatenated to form a new path. This collection of subpaths consist of a single path \(\eta(\lambda(0), \lambda(2^i))\) and \(2^k-1\) paths of the form:
    \[
        \eta(\lambda(j2^i), \lambda((j+1)2^i)),  \text{ where } i=n-k \text{ and } j \in \{1,\dots,2^k-1\}.
    \]
    
   Consider the base case \(k=1\). This is a single division into two paths. Let the paths \(\zeta^1, \zeta^2\) be defined as follows:
    \begin{align*}
        \zeta^1 &:= \eta\left(\lambda(0), \lambda(2^{n-1})\right) \\
        \zeta^2 &:= \eta\left(\lambda(2^{n-1}), \lambda(2^n)\right).
    \end{align*}
    
    By (G2), we have that the \(D\)-neighbourhood of \(\zeta^1 * \zeta^2\) contains \(\eta(\lambda(0), \lambda(2^n))\). After repeating \(n\) times, we have \(2^n\) paths given by 
    \begin{align*}
    \eta^1 &:= \eta\left(\lambda(0), \lambda(1)\right), \\
    &\vdots \\
    \eta^i &:= \eta\left(\lambda(i-1), \lambda(i)\right), \\
    &\vdots \\
    \eta^{2^n} &:= \eta\left(\lambda(2^n-1), \lambda(2^n)\right).
    \end{align*}

    \begin{figure}[H]
        \centering
        \includegraphics[width=0.70\textwidth]{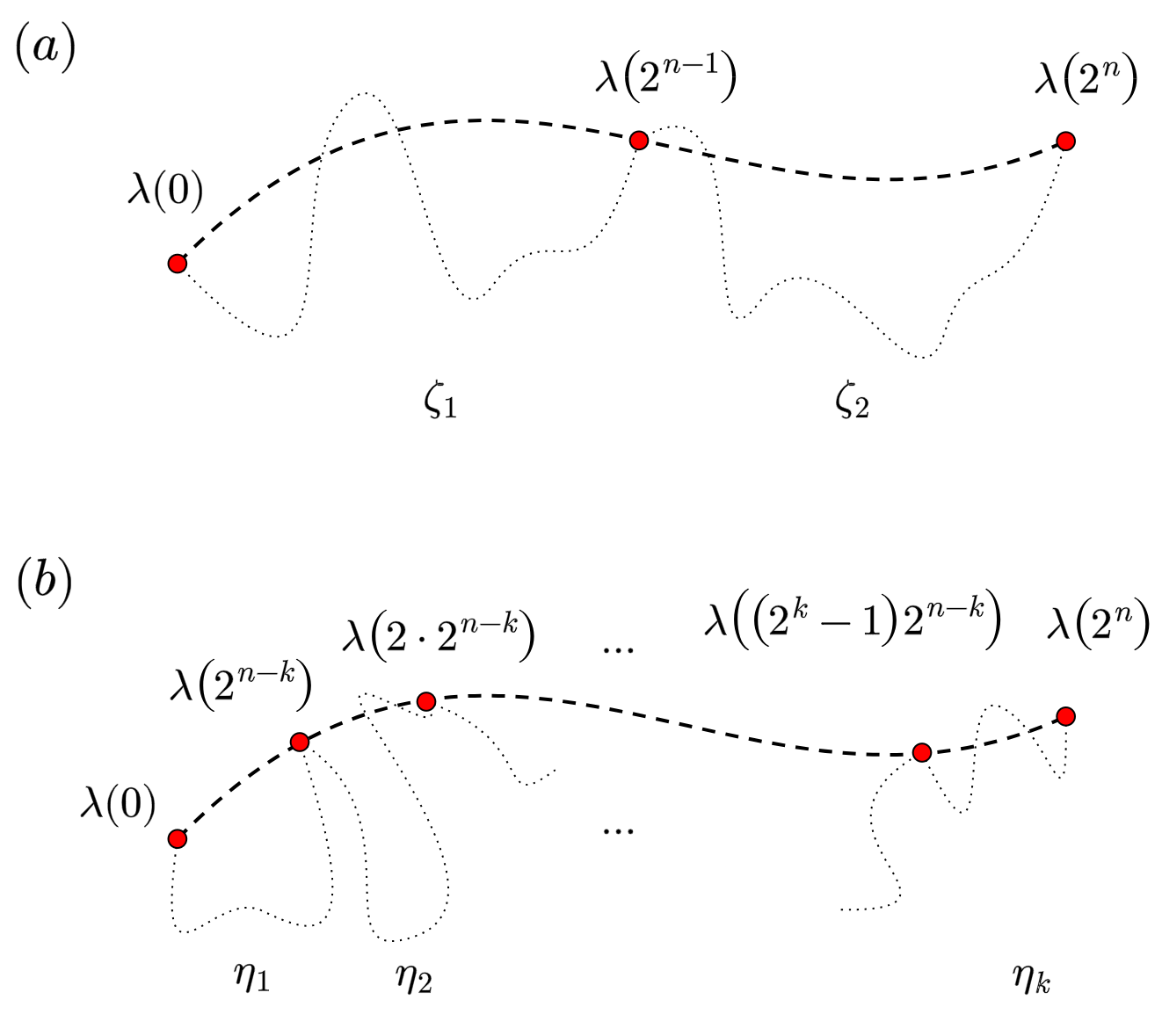}
        \caption{The \(D\)-coarsely connected paths in (a) arise from a subdivision of the interval \([0,2^n]\) forming two paths, and (b) shows a subdivision forming \(k\) paths.}
        \label{fig:figure1}
    \end{figure}
    
    By repeated application of (G2), we have that the \((nD)\)-neighbourhood of \(\eta^1 * \dots * \eta^n\) contains \(\eta(\lambda(0), \lambda(2^n))\). By (G1), each of the subpaths of the form \(\eta^k\) has diameter at most \(\frac{1}{2}(q_1+q_2)D\). Thus, \(\eta(\lambda(0), \lambda(2^n))\) is contained in the \((D\log_2\ell(\lambda) + \frac{1}{2}(q_1+q_2)D)\)-neighbourhood of \(\lambda\), since \(nD = D\log_2\ell(\lambda)\).
    
    \medskip
    \emph{Part 2:}
    Now let \(\gamma: [0,1] \to X\) be a \((q_1,q_2)\)--quasi-geodesic from \(x\) to \(y\).
    It is useful to define a time parameter for \(\eta_{xy}\) which maximises the distance \(d(\eta_{xy},\gamma)\), however as \(\eta_{xy}\) and \(\gamma\) are not necessarily continuous, such a distance may not be realised.

    Instead, let \(r := \mathrm{sup}_{\tau \in [0,1]}(d(\eta_{xy}(\tau),\gamma))\). For a fixed \(\eps >0\), consider the set
    \[
    \mathrm{Dist}_\eps := \left\{ T \in [0,1] \mid d(\eta_{xy}(T),\gamma) > r - \eps \right\}
    \]
    of time parameters so that \(r\) is almost achieved. By definition of the supremum, \(\mathrm{Dist}_\eps\) is non-empty. Therefore, we may fix some \(t \in \mathrm{Dist}_\eps\).

    We distinguish between two cases:
    \begin{enumerate}[label=\emph{Case~\arabic*:}, itemindent = 3em]
         \item If \(d(x, \eta_{xy}(t)) \leq q_1^2q_2+q_1^2r+q_2+r\), let \(t_1 = 0\).
        \item Otherwise, coarse connectivity implies that there exists \(t_1 < t\) such that \(d\left(\eta_{xy}(t_1), \eta_{xy}(t)\right) \in (q_1^2q_2+q_1^2r+q_2+r,q_1^2q_2+q_1^2r+q_2+r+D]\).
    \end{enumerate}
    We define \(t_2>t\) in the same way:
    \begin{enumerate}[label=\emph{Case~\arabic*:}, itemindent = 3em]
        \item If \(d(\eta_{xy}(t),y) \leq q_1^2q_2+q_1^2r+q_2+r\), let \(t_2 = 1\).
        \item Otherwise, coarse connectivity implies that there exists \(t_2 > t\) such that \(d\left(\eta_{xy}(t_2), \eta_{xy}(t)\right) \in (q_1^2q_2+q_1^2r+q_2+r,q_1^2q_2+q_1^2r+q_2+r+D]\).
    \end{enumerate}

    A natural question that arises at this point is why we select the constant \(q_1^2q_2+q_1^2r+q_2+r\). Roughly speaking, we would like to choose \(t_1\) and \(t_2\) so that they correspond to points on \(\eta_{xy}\) which are sufficiently far apart.
    
    This particular choice works out nicely in Part 4, but to get an idea we briefly illustrate where this constant appears. For \(i \in \{1,2\}\), there exist \(s_i\) such that \(d(\eta_{xy}(t_i), \gamma(s_i)) \leq r\) (if \(t_i \in \{0,1\}\), then take \(s_i=t_i\)). Observe that by the triangle inequality,
    \(
        d(\gamma(s_1), \gamma(s_2)) \leq 2(q_1^2q_2+q_1^2r+q_2+r) + 2r + 2D.
    \)
    We revisit this in Part 4.

    \medskip
    \emph{Part 3:}
    We return to the path \(\lambda\) defined in Part 1. As \(\lambda\) was defined as an arbitrary coarsely Lipschitz path (not necessarily with endpoints \(x\) and \(y\)), we may set \(\lambda\) as the specific path defined as follows.
    Let  \(\lambda : [0, \ell(\lambda)] \to X\) be the path obtained by concatenating three quasi-geodesic segments: a \((q_1,q_2)\)--quasi-geodesic \(\chi_1 = \eta(\eta_{xy}(t_1), \gamma(s_1))\), the subpath \(\gamma([s_1,s_2])\), and a \((q_1,q_2)\)--quasi-geodesic \(\chi_2 = \eta(\gamma(s_2), \eta_{xy}(t_2))\) (Fig.~\ref{fig:figure2}). 

    \begin{figure}[h!]
    \centering
    \includegraphics[width=0.75\textwidth]{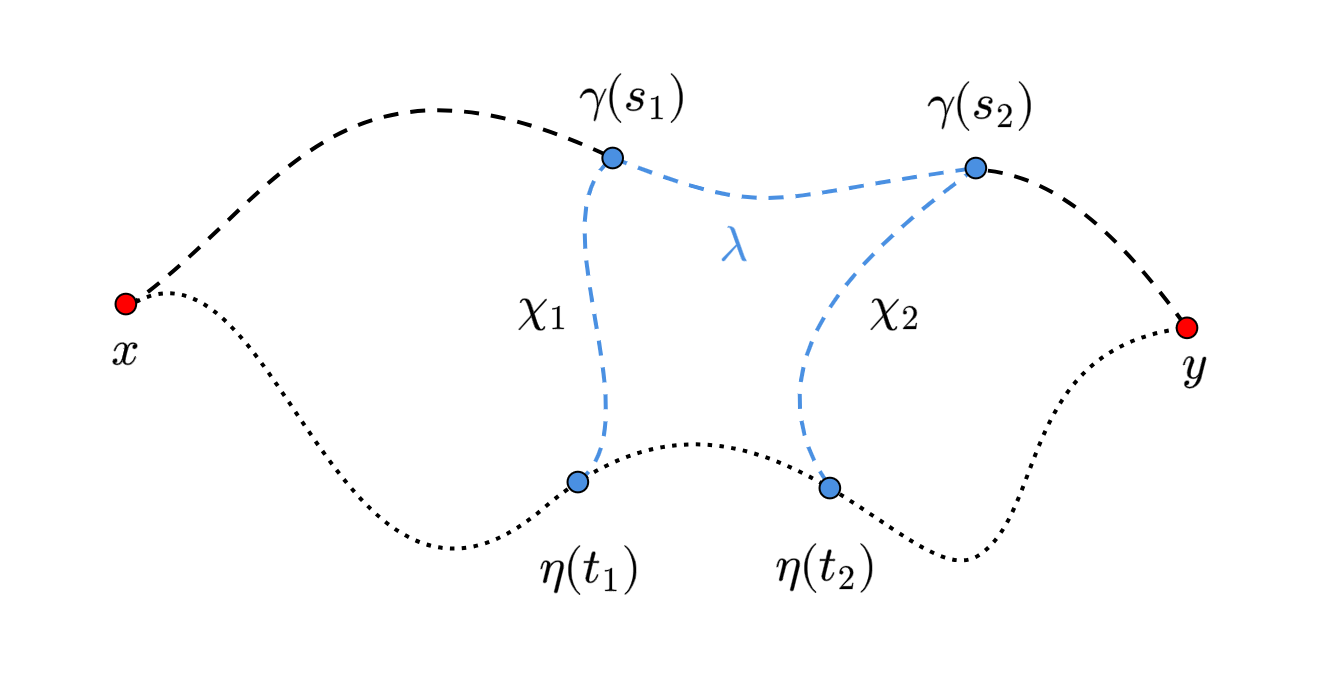}
    \caption{The path \(\lambda\) shown as a concatenation of three \((q_1,q_2)\)--quasi-geodesic paths.}
    \label{fig:figure2}
    \end{figure}
    
    We write this as
    \begin{align*}
        \lambda &= \eta(\eta_{xy}(t_1),\gamma(s_1)) * \gamma([s_1,s_2]) * \eta(\gamma(s_2),\eta_{xy}(t_2)), \\
        &= \chi_1 * \gamma([s_1,s_2]) * \chi_2.
    \end{align*}
    
    Since the path \(\lambda\) is a composition of three \((q_1,q_2)\)--quasi-geodesics,
    \[
    \ell(\lambda) = \ell(\chi_1) + (s_2-s_1) + \ell(\chi_2)
    \]
    and this means that
    \begin{align*}
        \frac{1}{q_1} \ell(\lambda) - q_2 &= \frac{1}{q_1}\left(\ell(\chi_1) + (s_2-s_1) + \ell(\chi_2)\right) - q_2 \\
        &\leq d(\eta_{xy}(t_1),\gamma(s_1)) + d(\gamma(s_1), \gamma(s_2)) \\
        &\quad +d(\gamma(s_2), \eta_{xy}(t_2)) +2q_2 \\
        &\leq 2(q_1^2q_2+q_1^2r+q_2+r) + 4r + 2D + 2q_2.
    \end{align*}
    Hence, we have that
    \begin{align*}
        \ell(\lambda) &\leq q_1\left(2(q_1^2q_2+q_1^2r+q_2+r) + 4r + 2D+3q_2\right) \\
        &\leq 8rq_1^3 + 7q_1^3q_2+2Dq_1.
    \end{align*}

    From the conclusion of Part 1, this implies that the 
    \[
        \left(D\log_2\left(8rq_1^3 + 7q_1^3q_2+2Dq_1\right) +\frac{1}{2}\left(q_1+q_2\right)D\right)\text{--neighbourhood}
    \]
    of \(\lambda\) contains \(\eta\left(\eta_{xy}(t_1),\eta_{xy}(t_2)\right)\). By (G2), \(\eta\left(\eta_{xy}(t_1),\eta_{xy}(t_2)\right)\) is at most Hausdorff distance of \(D\) from the subpath \(\eta_{xy}([t_1,t_2])\). Combining these two results yields that the 
    \[
    \left(D\log_2\left(8rq_1^3 + 7q_1^3q_2+2Dq_1\right)+\frac{1}{2}\left(q_1+q_2\right)D + D\right)\text{--neighbourhood}
    \]
    of \(\lambda\) contains the subpath \(\eta_{xy}([t_1,t_2])\). 
    
    \medskip
    \emph{Part 4:} 
    Let \(\gamma(s)\) be the point on \(\gamma\) where the distance \(r\) from \(\eta_{xy}\) is almost achieved, i.e. \(d(\eta_{xy}(t),\gamma) = d(\eta_{xy}(t), \gamma(s))\).
    We continue by making the following claim.
    
    \noindent
    \emph{Claim 1.}
    The below equality is satisfied:
    \[
        d(\eta_{xy}(t), \lambda)) = d(\eta_{xy}(t), \gamma(s)).
    \]

    In other words, \(\gamma(s)\) is the nearest point on \(\lambda\) to \(\eta_{xy}(t)\). 
    \begin{proof}[Proof of Claim 1.]
    Briefly, this claim holds because of our choice of the distance \(q_1^2q_2+q_1^2r+q_2+r\) we chose to be suitably large in Part 2. We expand on this as follows. Assume \(t_1>0\) and consider \(\chi_1 = \eta(\eta_{xy}(t_1), \gamma(s_1))\), one of the three \((q,q)\)--quasi-geodesics which form \(\lambda\). In order to show the claim, we want to show there is no point \(p_1 \in \chi_1\) such that \(d(\eta_{xy}(t), p_1) < d(\eta_{xy}(t), \gamma(s)) \leq r\), and we can form an equivalent statement for \(p_2 \in \chi_2\).

    As \(\chi_1: [a,b] \to X\) is a \((q_1,q_2)\)--quasi-geodesic,
    \[
        \frac{1}{q_1}|a-b| - q_2 \leq d(\eta_{xy}(t_1),\gamma(s_1)) \leq r.
    \]
    Rearranging gives
    \begin{align*}
        |a-b| &\leq q_1(q_2+r).
    \end{align*}

    Therefore for any point \(p_1 \in \chi_1\), we have
    \[
    d(\eta_{xy}(t_1),p_1) \leq q_1|a-b|+q_2 \leq q_1^2(q_2+r)+q_2,
    \]
    
    and by the triangle inequality, 
    \begin{align*}
    d(p_1, \eta_{xy}(t)) &\geq d(\eta_{xy}(t_1),\eta_{xy}(t)) - d(\eta_{xy}(t_1),p_1) \\
    &\geq (q_1^2q_2+q_1^2r+q_2+r) - (q_1^2q_2+q_1^2r+q_2) \\
    &=r.
    \end{align*}
    
    Hence, there is no \(p_1 \in \chi_1\) such that \(d(\eta_{xy}(t), p_1) < r\). The argument can be repeated for \(t_2>t\) with the points \(\eta_{xy}(t_2)\) and \(p_2\) instead of \(\eta_{xy}(t_1)\) and \(p_1\). Hence, the claim follows.
    \end{proof}

    We use the above claim to show that \(\eta_{xy}\) lies in a uniform neighbourhood of any \((q_1,q_2)\)--quasi-geodesic from \(x\) to \(y\), which we have denoted by  \(\gamma\). Since the closest point to \(\eta_{xy}(t)\) on \(\lambda\) is actually on the subpath \(\gamma([s_1,s_2])\) (the second subpath in the concatenation which defines \(\lambda\)), and \(\eta_{xy}(t) \in \eta_{xy}([t_1,t_2])\) is contained in the \((D\log_2\left(8rq_1^3 + 7q_1^3q_2 + 2Dq_1\right) +\frac{1}{2}(q_1+q_2)D + D)\)--neighbourhood of \(\lambda\), \(\eta_{xy}(t)\) is also in this neighbourhood of \(\lambda\). Hence, the constant \(r\) satisfies the inequality:
    \[
        r \leq D\log_2\left(8rq_1^3 + 7q_1^3q_2 + 2Dq_1\right) +\frac{1}{2}(q_1+q_2)D + D.
    \]

    We define the function \(f_{q_1,q_2}: \mathbb{R}_{\geq 0} \to \mathbb{R}_{\geq 0}\) by
    \[
    f_{q_1,q_2}(x) := D\log_2\left(8xq_1^3 + 7q_1^3q_2 + 2Dq_1\right) +\frac{1}{2}(q_1+q_2)D + D.
    \]
    This is the upper bound given on the right hand side above, where the constant \(r\) is instead replaced by a variable \(x\). Note that as \(x\) grows linearly, \(f\) grows logarithmically. Hence, the value \(f(x)\) is eventually overtaken by \(x\). In other words, there is some number \(z\) such that if \(x \geq z\), then \(x > f(x)\). Since \(r \leq f(r)\), \(r\) must be less than any choice of this constant \(z\). 
    
    In other words, the set \(Z = \{f(z) : \text{if } x \geq z \text{ then } f(x) < x\}\) is non-empty. 
    
    Since \(f\) takes values in the non-negative reals, then \(Z\) is bounded below and by completeness has an infimum. Let \(\kappa = \inf(Z)\). Since \(f\) is increasing, we have \(r \leq f(r) \leq \kappa = \inf(Z)\). In Section 2.2 we shall obtain a sequence of estimates for \(\inf(Z)\), and we refer to this family of estimates as a sequence \(\kappa_n\).

    In short, the constant \(\kappa\) is a universal constant \(\kappa(q_1,q_2,D)\) which serves as an upper bound of \(r\). This therefore means that \(\eta_{xy}\) lies in a uniform neighbourhood of any \((q_1, q_2)\)--quasi-geodesic from \(x\) to \(y\), as desired.

\medskip
\emph{Part 5:}
    Finally, it remains to show the converse of the previous result: every \((q_1,q_2)\)--quasi-geodesic \(\gamma : [0,n] \to X\) from \(x\) to \(y\) lies in a uniform neighbourhood of \(\eta_{xy}\). The set
    \[
        S = \{s \in [0,n]: \text{ there exists } t \text{ satisfying } d(\gamma(s), \eta_{xy}(t)) \leq \kappa\}
    \]
    is non-empty by the definition of the distance \(r\) and points \(s,t\) in Part 2. Given \(s \in [0,n]\), let \(t\) be maximal such that there exists \(s'' < s\) with \(d(\gamma(s''), \eta_{xy}(t)) \leq \kappa\). Consider \(\eta(\eta_{xy}(t),y)\). By coarse connectivity, there exists \(t' > t\) for which \(d(\eta_{xy}(t),\eta_{xy}(t')) \leq D\). Fix \(s' \in [0,n]\) such that \(d(\gamma(s'), \eta_{xy}(t')) \leq \kappa\). We may deduce that from the choice of \(t\), we have \(s'' \leq  s \leq s'\) and by the triangle inequality,
        \begin{align*}
            d(\gamma(s''), \gamma(s')) &\leq d(\gamma(s''), \eta_{xy}(t)) + d(\eta_{xy}(t),\eta_{xy}(t')) + d(\eta_{xy}(t'), \gamma(s')) \\
            &\leq 2\kappa + D.
        \end{align*}
        As \(\gamma\) is a \((q_1,q_2)\)--quasi-geodesic, this implies that \(|s''-s'| \leq q_1(2\kappa + D + q_2)\) and therefore
        \begin{align*}
            d(\gamma(s), \gamma(S)) &\leq q_1\frac{|s''-s'|}{2} + q_2 \\
                                    &\leq \frac{q_1^2}{2}(2\kappa + D + q_2) + q_2.
        \end{align*}
        
        Hence, \(\gamma(S)\) is \(\left(\frac{q_1^2}{2}(2\kappa + D + q_2) + q_2\right))\)-dense in \(\mathrm{Im}(\gamma)\). Therefore \(\gamma\) is contained in the \(\left(\frac{q_1^2}{2}(2\kappa + D + q_2) + q_2 + \kappa)\right)\)-neighbourhood of \(\eta_{xy}\), and so \(g_{q_1,q_2,D}(\kappa) = \frac{q_1^2}{2}(2\kappa + D + q_2) + q_2 + \kappa\).

    We recall from the start of the proof that the quasi-hyperbolicity constant \(\delta'\) satisfies \(\delta' = 2g_{q_1,q_2,D} + D\). Using the above expression for \(g_{q_1,q_2,D}\),
    \begin{align*}
        \delta'(q_1,q_2,D) = 2g_{q_1,q_2,D} + D &= 2\left(\frac{q_1^2}{2}\left(2\kappa + D + q_2\right) + q_2 + \kappa\right) + D.
    \end{align*}
\end{proof}

\subsection{Evaluation of \texorpdfstring{\(\kappa\)}{K}}
\label{section:kappa}
Here, we provide two approaches for deducing the universal constant \(\kappa\) defined in the proof of Proposition~\ref{prop:guessinggeodesics}. The first is by estimation, and the second is an explicit solution given by the Lambert \(W\) function. The \(W\) function is defined as the inverse to the function \(w \mapsto we^w\). We retain both methods, since our estimation is likely easier to use in practice.
The following lemma gives a sequence \(\kappa_n\) of estimates for \(\kappa = \inf(Z)\).
\begin{lem}
    \label{lem:kappa1}
    Let \(q := \max\{q_1,q_2\}\) and recall that:
\begin{equation*}
    \label{eq:f}
    f(x) = D\log_2\left(8xq_1^3 + 7q_1^3q_2 + 2Dq_1\right) +\frac{1}{2}\left(q_1+q_2\right)D + D
\end{equation*}

and \(Z = \{f(z) : \text{if } x \geq z \text{ then } f(x) < x\} \). 

If \(q, D \geq 1\), then there exists a sequence \(\{\kappa_n\}_{n \in \mathbb{N}} \in Z\) of the following form which estimates \(\kappa = \inf(Z)\): let \(\kappa_n := K_nqD^{1+\eps_n}\), where 
    \[
    K_n = \log_2(n+8)+9+\lceil 1+\log_2(\log_2(n+8)+9) \rceil
    \]
    and \(\eps_n = \frac{1}{\log_2(n+8)} \leq \frac{1}{3}\).
\end{lem}
\begin{proof}
Given \(n \in \mathbb{N}\), let \(\kappa_n = K_nqD^{1+\eps_n}\) where \(K_n\) and \(\eps_n\) are as given in the statement above. From this, we have:
\begin{align*}
    f(x) &\leq D\log_2\left(8q^3x + 7q^4 + 2qD\right) +qD + D
\end{align*}
Furthermore, since \(f\) is monotonically increasing, then the following holds whilst \(x \leq \kappa_n\):
\begin{align*}
f(x) &\leq D\log_2\left(17Kq^4D^{1+\eps_n}\right) + qD + D \\
     &\leq D\log_2(8Kq^4D^{1+\eps_n} + 7q^4 + 2qD) + qD + D.
\end{align*}

Next, observe that 
\begin{align*}
    \log_2(K_n) &= \log_2(\log_2(n+8)+9+ \left\lceil 1 + \log_2(\log_2(n+8)+9) \right\rceil) \\
              &\leq \log_2(2(\log_2(n+8)+9)) \\
              &\leq \lceil \log_2(2(\log_2(n+8)+9)) \rceil \\
              &= \log_2(n+8) \lceil 1 + \log_2(\log_2(n+8)+9) \rceil \\
              &\quad -(\log_2(n+8)+1) + 1\\
              &= K_n-8-\frac{1+\eps_n}{\eps_n}. 
\end{align*}

Therefore, we have that
\(
    \log_2(K_n) \leq K_n-8-\frac{1+\eps_n}{\eps_n}.
\) \hfill \inlinetag{eq:Kn}

Set \(T_n = q^{\frac{4\eps_n}{1+\eps_n}}D^{\eps_n} \geq 1\). We can derive two inequalities \eqref{eq:logsandk} from the previous one \eqref{eq:Kn}. To achieve the first inequality, we multiply the right hand side of the inequality above by \(T_n\) and the second follows from the condition \(\eps_n \leq \frac{1}{3}\).
\begin{equation*} \label{eq:logsandk}
    \log_2(K_n) + \frac{1+\eps_n}{\eps_n}T_n \leq (K_n-8)T_n \leq (K_n-8)qD^{\eps_n}. \tag{\(\ast \ast\)}
\end{equation*}

Using this, we have
\begin{align*}
    f(\kappa_n) &\leq D\log_2(17) + D\log_2(K_nq^4D^{1+\eps_n}) + qD + D \\
    &\leq D\log_2(K_nq^4D^{1+\eps_n}) + 7qD \\
    &\leq D\left(\log_2(K_n) + \frac{1+\eps_n}{\eps_n}\log_2\left(q^{\frac{4\eps_n}{1+\eps_n}}D^{\eps_n}\right)\right) + 7qD \\
    &\leq D\left(\log_2(K_n) + \frac{1+\eps_n}{\eps_n}q^{\frac{4\eps_n}{1+\eps_n}}D^{\eps_n}\right) + 7qD \\
    &\leq (K_n-8)qD^{1+\eps_n} + 7qD & \quad(\text{By } \eqref{eq:logsandk})\\
    &= K_nqD^{1+\eps_n} - 8qD^{1+\eps_n} + 7qD \\
    &\leq K_nqD^{1+\eps_n} = \kappa_n.
\end{align*}

We may conclude that for each universal constant \(\kappa_n\), if \(x\leq \kappa_n\) then \(f(x) \leq x\). We claim that in order for \(x\) to satisfy the reverse inequality \(x \leq f(x)\), we must have that \(x \leq \kappa_n\). To confirm this, we differentiate as below. As \(f\) is monotonically increasing,
\[
    f'(x) \leq \frac{d}{dx}\left(D\log_2(8q^3x+9q^4D)\right) 
    = \frac{8q^3D}{\ln(2) \cdot (8q^3x+9q^4D)}.
\]
If \(x \geq \kappa_n\), then
\[
    f'(x) \leq \frac{8q^3D}{\ln(2) \cdot (8K_nq^4D^{1+\eps_n}+9q^4D)} 
    \leq \frac{8q^3D}{\ln(2) \cdot (17K_nq^4D^{1+\eps_n})}
    <1.
\]
Therefore, for all \(x \geq \kappa_n\), we have \(f(x) < f(\kappa_n) + (x-\kappa_n)\). Since \(f(\kappa_n) \leq \kappa_n\), we have that if \(x \geq \kappa_n\), then \(f(x) < x\). In other words, if \(x \leq f(x)\), then \(x < \kappa_n\), as required.
\end{proof}
\begin{rem}
    It is likely that there is scope to improve the approximation steps above through different methods. In addition, taking \(q = \max\{q_1,q_2\}\) gives us an estimate which can be applied to a general \((q_1,q_2)\)--quasi-geodesic metric space. However, if we were to restrict to a \((1,q)\)--quasi-geodesic metric space at this stage we could improve our estimate.
\end{rem}

Next, we consider an explicit solution given by the Lambert \(W\) function.
\begin{defn}
    The \emph{Lambert \(W\) function} is the multivalued inverse of the function \(w \mapsto we^w\).
\end{defn}

If \(z=we^w\) is real, then for \(-1/e \leq z <0\) there are two possible real values of \(W(z)\). The branch satisfying \(W(z) \geq -1\) is often referred to as the \emph{principal branch}. 
\begin{lem}
\label{lem:lambert}
    We may rewrite \(f\) in the form
    \[f(x) = D\log_2(Ax+B)+C,\]
    where \(A = 8q_1^3, B=7q_1^3q_2+2Dq_1, C=\frac{1}{2}(q_1+q_2)D+D\). Recall that
    \(
    \kappa = \inf\{f(z) : \text{if } x \geq z \text{ then } f(x) < x\}.
    \)
    Then, \(\kappa\) is evaluated explicitly by:
    \[
    \kappa = \frac{D}{\ln(2)}W\left(-\frac{ln(2)}{DA}e^{\frac{-\ln(2)}{DA}(B+AC)}\right) - \frac{B}{A}
    \]
    where \(W\) refers to the principal branch of the Lambert \(W\) function.
\end{lem}
\begin{proof}   
    Recall that \(f\) is defined over and takes values in the non-negative reals. Define \(g: \mathbb{R}_{\geq 0} \to \mathbb{R}\) by \(g(x) := x-f(x)\). As \(q_1, q_2 \geq 0\) and \(D > 0\), \(g''(x)>0\) and therefore \(g\) is strictly convex. The function \(g\) therefore has at most two zeros. If \(g\) has no zeros, then as \(Z\) is bounded below, we take \(\kappa = 0\).

    We may therefore assume that \(g\) has one or two zeros. Let \(x_*\) denote the largest solution of \(g(x)=0\). As \(g\) is continuous, \(\kappa = \inf(Z) = f(x_*) = x_*\). The rest of the proof then follows by the below algebraic manipulation, where we put \(f\) in the form for the Lambert \(W\) function.

    It is helpful to define the following variables \(J= \frac{\ln(2)}{D}\) and \(y=Ax+B\).

    \begin{align*}
      x_* &= D\log_2(Ax+B)+C \\
    &\iff x-C =\frac{D}{\ln(2)}\ln(Ax+B) \\
    &\iff x-C =\frac{1}{J} \ln(Ax+B)  \\
    &\iff e^{J(x-C)} = Ax + B \\
    &\iff e^{J(\frac{y-B}{A} - C)} = y  \\
    &\iff ye^{-\frac{Jy}{A}} = e^{-\frac{J}{A}(B+AC)} \\
    &\iff -\frac{J}{A}y = W(-\frac{J}{A}e^{-\frac{J}{A}(B+AC)}) \\
    &\iff x_* = \frac{D}{\ln(2)}W\left(-\frac{ln(2)}{DA}e^{\frac{-\ln(2)}{DA}(B+AC)}\right) - \frac{B}{A}.
    \end{align*}    
\end{proof}

\subsection{Determining a hyperbolicity constant}
\label{section:quasi}
In this section, we discuss how we can relate a quasi-hyperbolicity constant, like the one determined in the previous section, to a hyperbolicity constant.

It is a fact that not all quasi-hyperbolic spaces are four-point hyperbolic: one example is the graph of the function \(y=\lvert x \rvert\) equipped with the subspace metric inherited from \(\mathbb{R}^2\) \cite{V}. But we can talk about a hyperbolicity constant for certain quasi-hyperbolic spaces. Specifically, it is known that roughly geodesic quasi-hyperbolic spaces are hyperbolic (see \cite[Proposition A.2]{PSZ}). However in the existing literature we were unable to find an explicit relationship between the quasi-hyperbolicity constant for roughly geodesic quasi-hyperbolic spaces and the hyperbolicity constant for hyperbolic spaces. Therefore, our objective in this section is to derive an explicit relationship between these two constants. 

To make calculations easier, we eliminate one of our variables by taking our metric space to be \({q}\)\emph{--roughly geodesic}, in other words, a \((1,q)\)--quasi-geodesic metric space. Hence from here on, we focus on \(q\)--rough geodesic triangles. In order to derive an explicit hyperbolicity constant, we use hyperbolicity in the sense of Gromov. 

\begin{defn}
    Let \(X\) be a metric space and \(\delta \in \mathbb{R}_{\geq 0}\). We say that \(X\) satisfies the \emph{four-point condition with constant} \({\delta}\) if 
        \[
        d(x,z) + d(y,w) \leq \max\{d(x,w)+ d(y,z), d(x,y) + d(z,w)\} + \delta
        \]
    for all \(x,y,z,w \in X\). If in addition \(X\) is roughly geodesic, then we say \(X\) is \(\delta\)--hyperbolic (in the sense of Gromov).
\end{defn}

Since we use the four-point condition definition of hyperbolicity, it is useful for us to have a notion of quasi-hyperbolicity in terms of quadrilaterals. Here we define a \(q\)--rough geodesic quadrilateral \(Q\) as a tuple \((\lambda_1,\lambda_2,\lambda_3,\lambda_4)\), where \(\lambda_i\) denotes a \(q\)--rough geodesic with endpoints \(x_i\) and \(x_{i+1}\) for \(i = 1,2,3\), and \(\lambda_4\) is a \(q\)--rough geodesic with endpoints \(x_4\) and \(x_1\).
\begin{defn}[Thin quadrilaterals]
    Let \(X\) be a \(q\)--roughly geodesic space and \(\lambda_1, \lambda_2, \lambda_3, \lambda_4\) be the sides of a quadrilateral \(Q \subset X\). Let \(\eps_T \geq 0\). We say that \(Q\) is \({\eps_T}\)\emph{--thin} if \(\lambda_i \subset \mathcal{N}_{\eps_T}(\lambda_{j} \cup \lambda_{k} \cup \lambda_{l})\), for all permutations \((i,j,k,l)\) of the indices \(\{1, \dots, 4\}\). 
\end{defn}
\begin{lem}
    \label{lem:thinquads}
    If all \(q\)--rough geodesic triangles are \(\delta'\)-thin, then all \(q\)--rough geodesic quadrilaterals are \(2\delta'\)-thin.
\end{lem}
\begin{proof}
    Let \(\lambda_1, \lambda_2, \lambda_3\) be three sides of a \(q\)--rough geodesic quadrilateral forming a path, where the sides are in respective order. Let \(\lambda_4\) be a \(q\)--rough geodesic such that \(\{\lambda_1, \dots, \lambda_4\}\) is the set of sides of a \(q\)--rough geodesic quadrilateral \(Q\). Let \(\gamma\) be a \(q\)--rough geodesic with endpoints \(\lambda_1(0), \lambda_3(0)\). Then,
    \(
        \gamma \subset \mathcal{N}_{\delta'}(\lambda_1 \cup \lambda_2) \text{ as these sides form a } q \text{-rough geodesic triangle.}
    \)
    From this, 
    \begin{align*}
        \lambda_4   &\subset \mathcal{N}_{\delta'}(\lambda_3 \cup \gamma) \\
                &\subset \mathcal{N}_{\delta'}\left(\lambda_3 \cup \mathcal{N}_{\delta'}(\lambda_1 \cup \lambda_2)\right) \\
                &\subset \mathcal{N}_{2\delta'}(\lambda_1 \cup \lambda_2 \cup \lambda_3).
    \end{align*}
    Therefore, an arbitrary \(q\)-rough quadrilateral is \(2\delta'\)-thin as required.
\end{proof}

Next we introduce a notion of bottlenecked rough geodesics and what it means for a pair of rough geodesics to fellow-travel. This will be a useful way for us to describe rough geodesics that stay suitably close over a long time. 
\begin{defn}
     Let \(\lambda_1, \lambda_2: [0,\ell(\lambda_i)] \to X\) be \(q\)--rough geodesics for \(i=1,2\), and let \(\lambda_1', \lambda_2'\) be subpaths of \(\lambda_1, \lambda_2\) respectively. The subpaths \(\lambda_1', \lambda_2'\) are \(b\)-bottlenecked if the Hausdorff distance between \(\lambda_1'\) and \(\lambda_2'\) is at most \(b\). The \(b\)-bottleneck length is the maximum of two distances: the distance between the endpoints of \(\lambda_1'\) and the distance between the endpoints of \(\lambda_2'\).
\end{defn}
\begin{defn}
\label{defn:fellowtravelling}
     We say that \(\lambda_1\) and \(\lambda_2\) \({\nu}\)\emph{--fellow-travel} if there exists a pair of subpaths which are \(\nu\)--bottlenecked. The \({\nu}\)\emph{--fellow-travelling distance} of \(\lambda_1\) and \(\lambda_2\), denoted \(\mathcal{D}_\nu(\lambda_1, \lambda_2)\), is the maximal \(\nu\)-bottleneck length over all possible pairs of \(\nu\)-bottlenecked subpaths. We denote the pair of subpaths which give this maximum distance by \(\lambda_1'\) and \(\lambda_2'\), and their endpoints by \(\lambda_1(a_1), \lambda_1(b_1)\) and \(\lambda_2(a_2), \lambda_2(b_2)\) for \(\lambda_1'\) and \(\lambda_2'\) respectively.
\end{defn}
\begin{lem}
    \label{lem:quasihyp}
    If \(X\) is a \(q\)--rough geodesic space that is \(\delta'\)--quasi-hyperbolic, then \(X\) satisfies the four-point condition with hyperbolicity constant \(\delta = \delta(q,\delta') = 56\delta' + 6q\).
\end{lem}
\begin{proof}
    Let \(x_1,x_2,x_3,x_4 \in X\). After relabelling the points, we may write
    \begin{align*}
        S &:= d(x_1, x_4) + d(x_2,x_3) \\
        M &:= d(x_1,x_2) + d(x_3,x_4) \\
        L &:= d(x_1, x_3) + d(x_2,x_4).
    \end{align*}
    with \(S \le M \le L\) similar to \cite[Part~III.H.1]{BH}. Consider a \(q\)--rough geodesic quadrilateral with sides \(\lambda_i\) from \(x_i\) to \(x_{i+1}\) (\(i=1,2,3\)) and \(\lambda_4\) from \(x_4\) to \(x_1\). 

    First, we consider a quadrilateral \(Q\) with sides labelled as above such that \(\lambda_2\) and \(\lambda_4\) \(\nu\)--fellow-travel. By the geometry of our method, making a choice of the Hausdorff distance \(\nu\) gives an expression for \(\delta\) depending on this choice of \(\nu\). We then choose \(\nu = 2\delta'\), since for any choice of four vertices in \(X\), there always exists a quadrilateral \(Q\) such that \(\lambda_2\) and \(\lambda_4\) \(2\delta'\)--fellow-travel.

    For the rest of this proof, we keep the notation \(\nu\) for clarity. The main result follows by two claims, Claim \hyperref[claim:2a]{\emph{2(a)}} and Claim \hyperref[claim:2b]{\emph{2(b)}}. We state and prove each in turn.
    
    \noindent
    \emph{Claim 2(a).} \label{claim:2a}
        If the sides \(\lambda_2\) and \(\lambda_4\) \(\nu\)--fellow travel, then \(D_{\nu}(\lambda_2,\lambda_4) \leq 5\nu\).
        
    \begin{proof}[Proof of Claim 2(a).]
        \begin{figure}[H]
        \centering
        \includegraphics[width=0.75\textwidth]{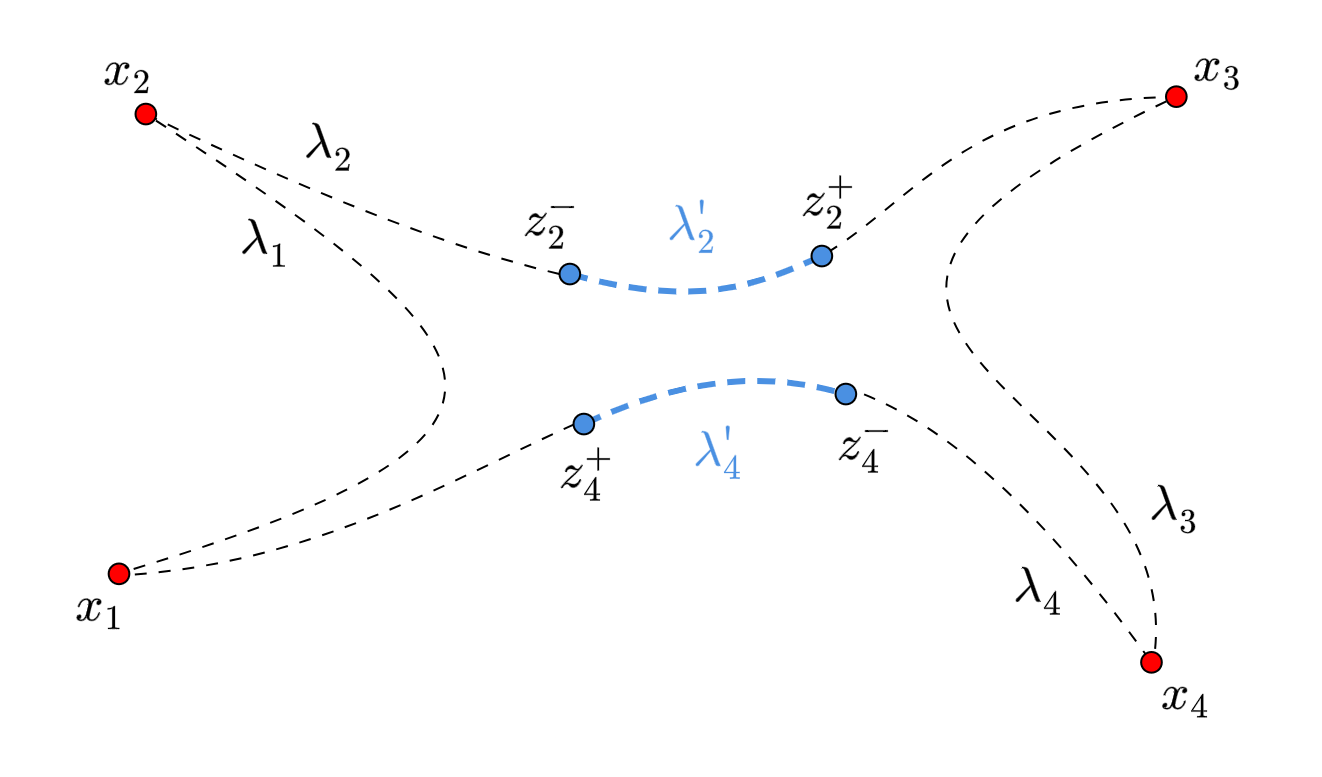}
        \caption{Dotted curves are used to denote quasi-geodesics. The blue curves denote the ``fellow-travelling section'', with the endpoints labelled accordingly.}
        \label{fig:figure3}
        \end{figure}
        For a contradiction, suppose that \(\lambda_2\) and \(\lambda_4\) \(\nu\)-fellow travel with fellow-travelling distance \(D_{\nu}(\lambda_2,\lambda_4) > 5\nu\).
        Let \(z_2^- = \lambda_2(a_1)\), \(z_2^+ = \lambda_2(b_1)\), \(z_4^- = \lambda_4(a_2)\) and \(z_4^+ = \lambda_4(b_2)\)
        where \(a_1, b_1, a_2\) and \(b_2\) are taken from the definition of fellow-travelling (Definition~\ref{defn:fellowtravelling}). Also, let \(\lambda_2^- := \lambda_2|_{[0, a_1]}\) and \(\lambda_2^+ := \lambda_2|_{[b_1, \ell(\lambda_2)]}\). Similarly define \(\lambda_4^-\) and \(\lambda_4^+\). Figure~\ref{fig:figure3} shows this setup.
            
        Then,
        \begin{align*}
            M &= d(x_1,x_2) + d(x_3,x_4) \\
              &\leq d(x_1,z_4^+) + d(z_4^+,z_2^-) + d(z_2^-,x_2) + d(x_3,z_2^+) + d(z_2^+,z_4^-) \\ &\quad +d(z_4^-,x_4) \\
              &\leq \ell(\lambda_2^-) + \ell(\lambda_2^+) +
               \ell(\lambda_4^-) + \ell(\lambda_4^+) + 2\nu + 4q.
        \end{align*}

        Also,
        \begin{align*}
            S &= d(x_1,x_4) + d(x_2,x_3) \\
              &\geq \ell(\lambda_2)-q + \ell(\lambda_4) - q \\
              &= \ell(\lambda_2^-) + \ell(\lambda_2^+) + \ell(\lambda_4^-) + \ell(\lambda_4^+) + \ell(\lambda_2') + \ell(\lambda_4') - 2q \\
              &\geq \ell(\lambda_2^-) + \ell(\lambda_2^+) + \ell(\lambda_4^-) + \ell(\lambda_4^+) + 2(D_{\nu}(\lambda_2,\lambda_4)-q) - 2q \\
              &> \ell(\lambda_2^-) + \ell(\lambda_2^+) + \ell(\lambda_4^-) + \ell(\lambda_4^+) + 2\nu + 4q.
        \end{align*}

        Therefore, by the above, 
        \[
        M \leq \ell(\lambda_2^-) + \ell(\lambda_2^+) + \ell(\lambda_4^-) + \ell(\lambda_4^+) + 2\nu + 4q < S
        \]
        and therefore, \(M < S\). But this contradicts our initial ordering of the lengths \(S,M\) and \(L\). Hence, we reach the desired contradiction and Claim \hyperref[claim:2a]{\emph{2(a)}} is proved.
    \end{proof}
    \noindent
    \emph{Claim 2(b).} \label{claim:2b} If the sides \(\lambda_2\) and \(\lambda_4\) \(\nu\)--fellow-travel, then \(d(\lambda_1,\lambda_3) < 14\delta' + 3\nu + 4q.\)
    \begin{proof}[Proof of Claim 2(b).]
      \begin{figure}[H]
        \centering
        \includegraphics[width=0.75\textwidth]{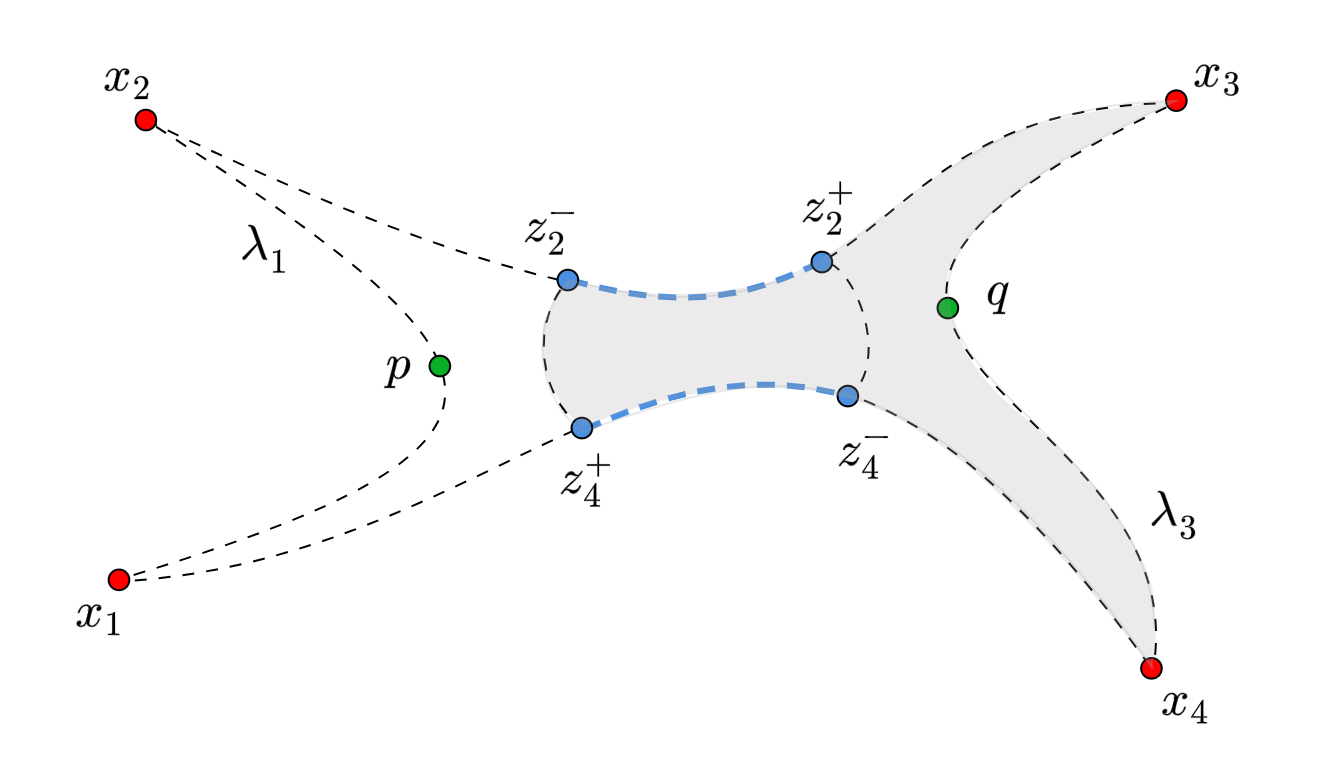}
        \caption{Quadrilateral separated into two thin quadrilaterals and a ``fellow-travelling section'' between them, where the latter two sections are greyed out.}
        \label{fig:figure4}
    \end{figure}
    Let \(z_2^-\) and \(z_4^+\) be as defined in the proof of Claim \hyperref[claim:2a]{\emph{2(a)}}. By rough geodesics and Claim \hyperref[claim:2a]{\emph{2(a)}}, \(d(z_2^-, z_2^+) \leq 5\nu + \delta'\). To upper bound the distance between the two opposite sides \(\lambda_1\) and \(\lambda_3\), we break down our starting quadrilateral into three parts: two thin quadrilaterals and a ``fellow-travelling section'' between them (Fig.~\ref{fig:figure4}). 
      
    We already have a bound for the length of the fellow-travelling section, so we concentrate on the thin quadrilateral made from the vertices \(x_1\), \(x_2\), \(z_2^-\) and \(z_4^+\) as shown.

    Suppose there exists a point \(y\) on \(\lambda_2\) at a distance of at least \(5\delta'+\nu\) from \(z_2^-\). Since \(d(z_2^-, z_4^+) \leq \nu\), the \(2\delta'\)-neighbourhood of the side of the quadrilateral with endpoints \(z_2^-\) and \(z_4^+\) has diameter at most \(4\delta'+\nu\). Hence, \(y\) cannot be \(2\delta'\)-close to this side. By the definition of fellow-travelling, the point \(y\) also cannot be \(2\delta'\)-close to the side on \(\lambda_4\). Applying thin quadrilaterals and Lemma~\ref{lem:quasihyp}, \(y\) is \(2\delta'\)-close to the side between \(x_1\) and \(x_2.\) In this case, we conclude that \(d(\lambda_1, z_2^-) \leq 7\delta'+\nu\).

    If such a point \(y\), does not exist, then \(d(\lambda_1, z_2^-) \leq 5\delta'+\nu\), which is also less than \(7\delta'+\nu\). Therefore we have that \(d(\lambda_1, z_2^-) \leq 7\delta'+\nu\). Repeating this argument for the second small quadrilateral gives the same result. 
    
    Therefore, 
    \begin{align*}
        d(\lambda_1, \lambda_3) &\leq d(\lambda_1, z_2^-) + d(z_2^-, z_2^+) + d(z_2^+,\lambda_3) \\
        &\leq 7\delta'+\nu + \nu + 4q + 7\delta'+\nu \\
        &= 14\delta' + 3\nu + 4q.
    \end{align*}
    \end{proof}
    
    Next we determine a four-point hyperbolicity constant using the result of Claim \hyperref[claim:2b]{\emph{2(b)}}. As before, let \(p,q\) be points on \(\lambda_1\) and \(\lambda_3\) respectively so that 
    \(
    d(\lambda_1,\lambda_3) = d(p,q).
    \)
    Then,
    \begin{align*}
        L &= d(x_1,x_3) + d(x_2,x_4) \\
          &\leq d(x_1,p) + d(p,q) + d(q,x_3) + d(x_2,p) + d(p,q) + d(q,x_4) \\
          &\leq \ell(\lambda_1) + \ell(\lambda_3) + 4q + 2(14\delta' + 7\nu) \\
          &= \ell(\lambda_1) + \ell(\lambda_3) + 28\delta' + 14\nu + 4q.
    \end{align*}

    Also,
    \begin{align*}
        M &= d(x_1,x_2) + d(x_3,x_4) \\
          &\geq \ell(\lambda_1)-q + \ell(\lambda_3)-q \\
          &= \ell(\lambda_1) + \ell(\lambda_3) - 2q.
    \end{align*}

    As a consequence of Lemma~\ref{lem:thinquads}, taking \(\nu = 2\delta'\) ensures that for every \(4\)-tuple of points, \(\lambda_2\) and \(\lambda_4\) \(\nu\)--fellow-travel. 
    Hence, \(L \leq M + 28\delta' + 14\nu + 6q = M + 56\delta' + 6q\). It follows that
    \(
        \delta = 56\delta' + 6q.
    \)
\end{proof} 

To conclude this section, we restate Theorem~\ref{thmx:mainresult} and provide an overview of the proof.
\begin{thmx-non}[Quantitative Guessing Geodesics]
    Let \(X\) be a \((1,q)\)--quasi-geodesic metric space.
    Assume that for some constant \(D > 0\) there are \(D\)-coarsely connected paths \(\eta_{xy}  = \eta(x,y) : [0,1] \to X\) from \(x\) to \(y\), for each pair \(x,y \in X\). Suppose \(X\) also satisfies the following conditions:
    \begin{enumerate}[label=(G\arabic*)]
        \item The diameter \(\mathrm{diam}(\eta_{xy}) \leq \frac{D d(x,y)}{2}\),
        \item For any \(s \leq t\), \(d_{\mathrm{Haus}}(\eta_{xy}[s,t], \eta(\eta_{xy}(s),\eta_{xy}(t)) \leq D\),
        \item For any \(x,y,z \in X\), \(\eta_{xy} \subset \mathcal{N}_D(\eta_{xz} \cup \eta_{zy})\).
    \end{enumerate}
    
    Then every \((1,q)\)--quasi-geodesic triangle in \(X\) is \(\delta'\)-thin, and \(X\) is \((56\delta' + 6q)\)--hyperbolic, where \(\delta' \leq 72qD^{\frac{5}{4}} + 2D + 3q\).
\end{thmx-non}
\begin{proof}
    By our guessing geodesics theorem, (Proposition~\ref{prop:guessinggeodesics}), \(X\) is quasi-hyperbolic with constant \(\delta' = \delta'(q_1,q_2,D) = 2\left(\frac{q_1^2}{2}(2\kappa + D + q_2) + q_2 + \kappa\right) + D\). As we consider \((1,q)\)--quasi-geodesics, this gives \(\delta' = 4\kappa + 2D + 3q\).

    By the proof of Proposition~\ref{prop:guessinggeodesics}, we in fact have a family of values for \(\kappa\), which we call \(\kappa_n\) for \(n \in \mathbb{N}\). We proceed by taking close to the infimum of the \(\kappa_n\) for the value of \(\kappa\). By Lemma~\ref{lem:kappa1}, \(\kappa_n = K_nqD^{1+ \frac{1}{log_2(n+8)}}\), where \(K_n\) is bounded by \(\log_2(n+8)+9+\lceil 1+\log_2(\log_2(n+8)+9) \rceil\). Taking \(n=8\) gives \(\kappa \leq \kappa_{8} = 18qD^\frac{5}{4}\). Hence, we have \(\delta' \leq 72qD^\frac{5}{4} + 2D + 3q\). Finally, by Lemma~\ref{lem:quasihyp}, \(X\) is hyperbolic with constant \(56\delta' + 6q\) by the four-point condition.
\end{proof}

\section{Application to the Curtain Model}
We devote the rest of this paper to an application of the quantitative guessing geodesics theorem which we discussed in the previous section. 

From here on, we work in the setting of the curtain model of \(\mathrm{CAT}(0)\) spaces introduced by Petyt--Spriano--Zalloum \cite{PSZ}. The curtain model, which we define later in this section, is a hyperbolic model of \(\mathrm{CAT}(0)\) spaces. Applying the quantitative guessing geodesics theorem allows us to estimate the hyperbolicity constant of this model explicitly and therefore how hyperbolic the space is.

An outline of our approach for this is as follows. We determine an numerical value for \(\kappa_n\) for any value of \(n \in \mathbb{N}\), which can be done once the values for \(q\) and \(D\) in the curtain model are known. These values for \(q,D\) will be calculated in Section~\ref{section:constants}. In Section~\ref{section:final}, the values of \(q,D\) and \(\kappa_n\) will determine a quasi-hyperbolicity constant which arises from the proof of the quantitative guessing geodesics theorem. Using the relation between the quasi-hyperbolicity and hyperbolicity constants derived in Section~\ref{section:quasi}, this in turn will provide us with a hyperbolicity constant.

\subsection{Background}
\label{section:curtain}
In this section, we follow the construction of the curtain model \((X,\mathsf{D})\) as established in \cite{PSZ}, and directly quote definitions from this source. For this, we require the definition of a curtain and chains of curtains which form the main components of the construction of this model.

\begin{defn}[{\cite[Definition~2.1]{PSZ}}]
    Let \(X\) be a \(\mathrm{CAT}(0)\) space, and \(\alpha: I \to X\) be a geodesic in \(X\). Let \(\pi_\alpha:X \to \alpha\) denote the closest point projection of \(X\) onto the geodesic \(\alpha\). For a number \(r\) with \(\left[r-\frac{1}{2},r+\frac{1}{2}\right]\) in the interior of \(I\), the \emph{curtain dual to} \({\alpha}\) \emph{at} \({r}\) is \(h_{\alpha,r} = \pi_\alpha^{-1}\left(\alpha\left(\left[r-\frac{1}{2},r+\frac{1}{2}\right]\right)\right).\)
\end{defn}
\begin{defn}[{\cite[Definition~2.2]{PSZ}}]
    Let \(X\) be a \(\mathrm{CAT}(0)\) space, and let \(h = h_{\alpha,r}\) be a curtain. The halfspaces determined by \(h\) are given by \(h^- = \pi_{\alpha}^{-1}\left(\alpha\left(I \cup \left(-\infty, r-\frac{1}{2}\right)\right)\right)\) and \(h^+ = \pi_{\alpha}^{-1}\left(\alpha\left(I \cup \left(r+\frac{1}{2}, \infty\right)\right)\right)\). Note that \(\{h^-, h, h^+\}\) is a partition of \(X\). If \(A\) and \(B\) are subsets of \(X\) such that \(A \subseteq h^-\) and \(B \subseteq h^+\), then we say that \(h\) \emph{separates} \(A\) from \(B\).
\end{defn}
\begin{defn}[{\cite[Definition~2.9]{PSZ}}]
    A \emph{chain of length} \({n}\) is a sequence of curtains \(c = \{h_i\}_{i=1}^n\) such that \(h_{i}\) separates \(h_{i-1}\) and \(h_{i+1}\) for all \(i = 1,\dots n-1\). We denote the length of a chain \(c\) by \(|c| = n\). We say that a chain \(\{h_i\}_{i=1}^n\) separates \(A,B \subseteq X\) if \(h_i\) separates \(A\) from \(B\) for all \(i=1, \dots, n\). The \emph{chain distance} between points \(x\) and \(y\), for \(x\) and \(y\) distinct points, is defined by
    \[
        d_{\infty}(x,y) = \max\{|c|: c \text{ is a chain separating } x \text{ and } y\} + 1.
    \]
\end{defn}
\begin{defn}[{\cite[Definition~2.11]{PSZ}}]
    Given \(L \in \mathbb{N}\), two disjoint curtains \(h_1, h_2\) are \({L}\)\emph{--separated} if any chain which meets both has length at most \(L\).
    We call a chain in which every adjacent pair of curtains is \(L\)-separated an \({L}\)\emph{--chain}.
\end{defn}

We may use this idea of a pairwise \(L\)-separated chain of curtains to form a notion of distance within a \(\mathrm{CAT}(0)\) space \(X\). In order to do this, we define a family of metric spaces \(X_L := (X, d_L)\) corresponding to \(L \in \mathbb{N}\).
\begin{defn}[{\cite[Definition~2.15]{PSZ}}]
    Given \(x,y \in X\), and \(L \in \mathbb{N}\), we define the \({L}\)\emph{--metric} by
    \[
    d_L(x,y) = 
    \begin{cases}
        \max\left\{ |c| : c \text{ is an } L\text{-chain sep. } x \text{ and } y \right\} + 1 &\quad \text{ if } x \neq y; \\
        0 &\quad \text { if } x=y.
    \end{cases}
    \]
\end{defn}
\begin{rem}
    The \(L\)-metric \(d_L\) can be shown to indeed be a metric in the usual way.  In addition, one can check that \(d_L(x,y) \leq d_{\infty}(x,y) \leq 1+d(x,y)\) for all \(L \in \mathbb{N}\).
\end{rem}

Therefore, we may think of these \(L\)-metrics as a family given by \(d_L\), where \(L\) takes values in \(\mathbb{N}\), with corresponding spaces \(X_L = (X, d_L)\). We combine these metrics into a single hyperbolic metric space \((X,\mathsf{D})\) as in \cite[Ch.~9]{PSZ}. This metric space is what is referred to as the curtain model of \(X\) throughout this paper, and is defined as below.
\begin{defn}
\label{defn:curtain}
    The \emph{curtain model of} \(X\), denoted by \((X, \mathsf{D})\), is defined by
    \(\mathsf{D}(x,y) = \sum_{L=1}^\infty \lambda_Ld_L(x,y)\), where \(\lambda_L\) is a sequence of numbers in \((0,1)\) such that 
    \[
    \sum_{L=1}^\infty \lambda_L < \sum_{L=1}^\infty L\lambda_L < \sum_{L=1}^\infty L^2\lambda_L =: \Lambda < \infty.
    \]
\end{defn}

By rescaling, we may assume that \(\Lambda := \sum_{L=1}^\infty L^2\lambda_L = 1\). Proceeding in this way will simplify various equations in the following section, so from here on we take \(\Lambda = 1\). 

\subsection{Determining constants for \texorpdfstring{\(\mathrm{CAT}(0)\)}{CAT(0)} spaces}
\label{section:constants}
Here we establish concrete values of \(q\) and \(D\) (as in \S2) for the curtain model of a \(\mathrm{CAT}(0)\) space.

To start, we require the definition of a reparametrisation map and an unparametrised rough geodesic. Recall from Section~\ref{section:quasi} that a \({q}\)\emph{--rough geodesic} is a \((1,q)\)--quasi-geodesic.

\begin{defn}
    A weakly increasing map \(r:I \to J\) is a \emph{reparametrisation map} for intervals \(I,J \subset \mathbb{R}\).
\end{defn}
\begin{defn}
    A path \(P:J \to X\) is an \emph{unparametrised rough geodesic} if there exists a reparametrisation map \(r : I \to J\) such that \(P \circ r\) is a rough geodesic.
\end{defn}
\begin{prop}[{\cite[Proposition~9.5]{PSZ}}]
    Let \(X\) be a \(\mathrm{CAT}(0)\) metric space. Then geodesics in \(X\) are unparametrised \(q\)–rough geodesics of the curtain model \(X_{\mathsf{D}}\), where \(q = \max\{6\Lambda,1\}+1\).
\end{prop}
\begin{proof}
    According to the proof provided in \cite{PSZ}, for \(x,y,z\) in a \(\mathrm{CAT}(0)\) geodesic, we have
    \begin{equation*}
    \label{eq:latex}
        \mathsf{D}(x,y) \geq \mathsf{D}(x,z) + \mathsf{D}(z,y) - 6\Lambda \hfill \tag{*}
    \end{equation*}
    where \(z \in [x,y]\) and \(\mathsf{D}\) is the metric defined in the curtain model. We expand on this proof to determine the value of \(q\).
    
    Let \(\alpha: I \to X\) be a geodesic in a \(\mathrm{CAT}(0)\) space \(X_{\mathsf{D}}\). We define maps as per the diagram below, where we aim to show that the image of \(P\) is an unparametrised rough geodesic in \(X_{\mathsf{D}}\), \(Q\) is a rough geodesic in \(X_{\mathsf{D}}\), and \(r:\mathbb{R} \to \mathbb{R}\) is a reparametrisation map. The map \(P\) is defined as the composition of an isometric embedding \(J \to (X, \mathsf{d})\), and the map into the curtain model.
    \begin{center}
        \begin{tikzcd}
            J \arrow[rd, "P"] \arrow[r, "\alpha"] &
            (X,\mathsf{d}) \arrow[d, "\mathrm{id}"] \\
            I \arrow[u,"r"] \arrow[r,"Q"] &
            (X,\mathsf{D})
        \end{tikzcd}
    \end{center}
    
     To prove the proposition, we show the following.
    \begin{enumerate}
        \item The image of the map \(Q: \mathbb{R} \to (X,\mathsf{D})\) is a rough geodesic in \(X_{\mathsf{D}}\).
        \item A reparametrisation map \(r:\mathbb{R} \to \mathbb{R}\) exists such that the diagram commutes.
        \item The set \(\mathrm{Im}(Q)\) is coarsely dense in \(\mathrm{Im}(P)\).
    \end{enumerate}

    Without loss of generality, assume \(0 \in I\) where \(\alpha: I \to X\). Let \(Q(0) = \alpha(0)\). For \(t>0\), let \(Q(t) \in P([0,\infty))\) such that \(\mathsf{D}(Q(0),Q(t)) \in [t,t+1]\). Similarly, for \(t < 0\), let \(Q(t) \in P((\infty, 0])\) satisfy \(\mathsf{D}(Q(t),Q(0)) \in [-t,-t+1]\). The point \(Q(t)\) exists for all \(t\), because \(\alpha\) is a geodesic and
    \[
     \mathsf{D}(x,y) = \sum_{L=1}^\infty \alpha_L d_L(x,y) \leq \sum_{L=1}^\infty \alpha_L\left(1+ \mathsf{d}(x,y)\right) \leq 1 + \mathsf{d}(x,y).
    \]

    We consider the first statement. Without loss of generality, let \(t_1 \leq t_2\) for \(t_1, t_2 \in \mathbb{R}\). It suffices to show that 
    \[
        \lvert\mathsf{D}(Q(t_1),Q(t_2)) - |t_1-t_2|\rvert \leq q
    \]
    where \(q\) is to be determined. There are three possible cases given by the order of \(0, t_1\) and \(t_2\).
    
    \medskip
    \emph{Case 1:} \(0 \leq t_1 \leq t_2\). Note that
       \begin{align*} 
        \mathsf{D}(Q(0),Q(t_1)) + \mathsf{D}(Q(t_1),Q(t_2)) &\leq \mathsf{D}(Q(0),Q(t_2)) + 6\Lambda \hfill \text{ by \eqref{eq:latex}}\\
                      &\leq t_2+1+6\Lambda. 
        \end{align*}
    Therefore, we can determine an upper bound by
        \begin{align*}
        \mathsf{D}(Q(t_1),Q(t_2)) &\leq t_2+1+6\Lambda - \mathsf{D}(Q(0),Q(t_1)) \\
                                  &\leq (t_2-t_1)+6\Lambda +1.
        \end{align*}
    For the lower bound, we apply the triangle inequality to give
        \begin{align*}
            \mathsf{D}(Q(t_1),Q(t_2)) &\geq \mathsf{D}(Q(0),Q(t_2)) - \mathsf{D}(Q(0),Q(t_1)) \\
            &\geq t_2 - (t_1+1).
        \end{align*}
        Thus, 
        \(
             \mathsf{D}(Q(t_1),Q(t_2)) \approx t_2-t_1,
        \)
        up to a constant error of \(q=\max\{6\Lambda + 1,1\} = 6\Lambda +1\).
    
    \medskip    
    \emph{Case 2:} \(t_1 \leq 0 \leq t_2\). As before, 
    \begin{align*} 
        \mathsf{D}(Q(t_1),Q(t_2)) &\geq \mathsf{D}(Q(t_1),Q(0)) + \mathsf{D}(Q(0),Q(t_2)) - 6\Lambda \hfill \text{ by \eqref{eq:latex}}\\
                      &\geq (t_2-t_1) - 6\Lambda. 
    \end{align*}
    The upper bound for \(\mathsf{D}(Q(t_1),Q(t_2))\) is 
    \begin{align*}
        \mathsf{D}(Q(t_1),Q(t_2)) &\leq \mathsf{D}(Q(t_1),Q(0)) +\mathsf{D}(Q(0),Q(t_2)) \\
        &\leq (-t_1+1)+(t_2+1) = (t_2-t_1) + 2.
    \end{align*}
    In this case, 
    \(
             \mathsf{D}(Q(t_1),Q(t_2)) \approx t_2-t_1,
        \)
        up to a constant error of \(q=\max\{6\Lambda, 2\}\).

    \medskip
    \emph{Case 3:} \(t_1 \leq t_2 \leq 0\).
        Working similar to Case 1 gives that \(\mathsf{D}(Q(t_1),Q(t_2)) \geq (t_2-t_1)-1\).
 
    We conclude that \(q=\max\{6\Lambda, 1\}+1\) satisfies all three cases. It remains to show the second statement, that there exists a reparametrisation \(r\) such that \(Q = P \circ r\).

    To do this, we wish to construct \(r'\) such that \(P = Q \circ r'\). We proceed with the construction of \(r'\) as follows. We have that \(P(0) = Q(0)\). Given \(s \in \mathbb{R}\), let \(t = \mathsf{D}(P(0), P(s))\). 
    From our construction of \(Q\), there exists \(t' \in \mathbb{R}\) with the same sign as \(s\) such that 
    \begin{equation*}
    \tag{**}
    \label{eqn:r'}
        \mathsf{D}(Q(0),Q(t')) \in [t,t+1].
    \end{equation*}
    
    This is because by the definition of \(Q\), we have that there exists at least one possibility for \(t'\) given by \(t'=t\). We construct \(r'\) by setting \(r'(s)\) equal to the minimal value of \(t'\) which satisfies \eqref{eqn:r'}.
    We claim that the distance \(\mathsf{D}(P(s),Q(t'))\) is uniformly bounded.
    
    To show this, note that by the definition of \(Q\), there exists \(s'\) such that \(Q(t') = P(s')\). If \(s < s'\), by \eqref{eq:latex} we have 
    \begin{align*}
     \mathsf{D}(P(0),P(s)) + \mathsf{D}(P(s),P(s')) - 6\Lambda &\leq \mathsf{D}(P(0),P(s')) \\
     &= \mathsf{D}(Q(0),Q(t')) \in [t,t+1].
    \end{align*}
    Therefore, 
    \[
    t + \mathsf{D}(P(s),P(s')) - 6\Lambda \leq t+1,
    \text{ and so }
    \mathsf{D}(P(s),Q(t')) \leq 6\Lambda +1.
    \]

    Otherwise, if \(s \geq s'\),
    \begin{align*}
    t + \mathsf{D}(P(s'), P(s)) - 6\Lambda &\leq 
    \mathsf{D}(P(0), P(s')) + \mathsf{D}(P(s'), P(s)) - 6\Lambda \\ &\leq \mathsf{D}(P(0), P(s)) = t.
    \end{align*}
    Thus, \(\mathsf{D}(P(s),Q(t')) \leq 6\Lambda\). Combining both cases implies that for all \(s \in \mathbb{R}\), \(\mathsf{D}(P(s),P(s')) \leq 6\Lambda + 1\). Therefore  \(\mathsf{D}(P(s),P(s'))\) is uniformly bounded by \(6\Lambda+1\) and as \(Q(t') = P(s')\), \(\mathrm{Im}(Q)\) is \((6\Lambda+1)\)-dense in \(\mathrm{Im}(P)\). Since \(Q\) is a \((6\Lambda +1)\)-rough geodesic, this implies that \(P\) is an unparametrised \((2(6\Lambda +1))\)-rough geodesic, as required.
\end{proof}

Therefore, explicit values for the constants \(q\) and \(D\) are given by \(q = \max\{6\Lambda,1\}+1\) and \(D=125\Lambda\), where the latter is obtained by Lemma~9.9 in \cite{PSZ}. Since we are taking \(\Lambda=1\), we have that \(q=7\) and \(D=125\).

\subsection{Hyperbolicity of the curtain model}
\label{section:final}
Finally, we demonstrate a sample application of the quantitative guessing geodesics theorem to estimate the hyperbolicity constant of the curtain model.

As discussed in Section~\ref{section:curtain}, we take the scaling constant \(\Lambda = 1\) for the curtain model \((X, \mathsf{D})\) introduced in Definition~\ref{defn:curtain}. We deduce the following estimates for its hyperbolicity constant, using various approaches as detailed below.

Our main guessing geodesics result, Theorem~\ref{thmx:mainresult}, tells us that our quasi-hyperbolicity constant is \(\delta' \leq 72qD^\frac{5}{4} + 2D + 3q\) and our hyperbolicity constant is \(56\delta + 6q\). Substituting the values of \(q=7\) and \(D=125\) for the curtain model derived in the previous subsection gives that \(\delta' \leq 2.11 \times 10^5\), and therefore we deduce a hyperbolicity constant of \(\delta \leq 1.19 \times 10^7\). (Using this method involves the auxiliary constant \(\kappa\) implicitly. In particular, we use the family of estimates \(\kappa_n\) for \(\kappa\) determined in Section~\ref{section:kappa} and take \(\kappa = \kappa_8 = 5.27 \times 10^4\).)

One may also determine, by plotting how \(\kappa_n\) varies with \(n\), that taking \(n = 2000\) yields a rough estimate for \(\kappa = \inf(Z)\). This gives \(\kappa = \kappa_{2000} = 3.40 \times 10^4\). Applying \(\delta' = 4\kappa + 2D + 3q\) with this reduced value of \(\kappa\) yields a quasi-hyperbolicity constant of \(\delta' \leq 1.37 \times 10^5\) and by Lemma~\ref{lem:quasihyp}, a hyperbolicity constant of \(\delta \leq 7.64 \times 10^6\).

To evaluate these estimates for the curtain model, we compare the estimates derived by various values of \(\kappa_n\) to the values determined by substituting in the constants for the curtain model directly into the definition of \(f\) (where \((q_1,q_2) = (1,q)\)). In doing so, we may deduce that \(\kappa \leq 2.41 \times 10^3\). 
This results in a quasi-hyperbolicity constant of \(\delta' \leq 9.92 \times 10^3\) and a hyperbolicity constant of \(\delta \leq 5.56 \times 10^5\).

\bibliographystyle{amsalpha}
\bibliography{main}

@book{Loh,
	address = {Cham},
	series = {Universitext},
	title = {Geometric {Group} {Theory}},
	copyright = {http://www.springer.com/tdm},
	isbn = {978-3-319-72253-5 978-3-319-72254-2},
	url = {http://link.springer.com/10.1007/978-3-319-72254-2},
	doi = {10.1007/978-3-319-72254-2},
	urldate = {2026-09-01},
	publisher = {Springer International Publishing},
	author = {Löh, Clara},
	year = {2017},
}

@book{BH,
	address = {Berlin, Heidelberg},
	series = {Grundlehren der mathematischen {Wissenschaften}},
	title = {Metric {Spaces} of {Non}-{Positive} {Curvature}},
	volume = {319},
	copyright = {https://www.springernature.com/gp/researchers/text-and-data-mining},
	isbn = {978-3-642-08399-0 978-3-662-12494-9},
	url = {https://link.springer.com/10.1007/978-3-662-12494-9},
	doi = {10.1007/978-3-662-12494-9},
	urldate = {2026-09-01},
	publisher = {Springer},
	author = {Bridson, Martin R. and Haefliger, André},
	year = {1999},
}

@article{B,
	journal = {Journal für die reine und angewandte Mathematik},
	title = {Intersection numbers and the hyperbolicity of the curve complex},
	volume = {2006},
	copyright = {De Gruyter expressly reserves the right to use all content for commercial text and data mining within the meaning of Section 44b of the German Copyright Act.},
	issn = {1435-5345},
	url = {https://www.degruyterbrill.com/document/doi/10.1515/CRELLE.2006.070/html?lang=en},
	doi = {10.1515/CRELLE.2006.070},
	number = {598},
	urldate = {2026-09-01},
	publisher = {De Gruyter},
	author = {Bowditch, Brian H.},
	month = sep,
	year = {2006},
	pages = {105--129},
}

@article{DGO,
	title = {Hyperbolically embedded subgroups and rotating families in groups acting on hyperbolic spaces},
	volume = {245},
	copyright = {https://www.ams.org/publications/copyright-and-permissions},
	issn = {1947-6221, 0065-9266},
	url = {https://www.ams.org/memo/1156/},
	doi = {10.1090/memo/1156},
	number = {1156},
	urldate = {2026-09-01},
	journal = {Memoirs of the American Mathematical Society},
	author = {Dahmani, F. and Guirardel, V. and Osin, D.},
	month = jul,
	year = {2016},
    pages = {v+152},
}

@article{DS,
	title = {Tree-graded spaces and asymptotic cones of groups},
	volume = {44},
	copyright = {https://www.elsevier.com/tdm/userlicense/1.0/},
	issn = {00409383},
	url = {https://linkinghub.elsevier.com/retrieve/pii/S0040938305000297},
	doi = {10.1016/j.top.2005.03.003},
	number = {5},
	urldate = {2026-09-01},
	journal = {Topology},
	author = {Druţu, Cornelia and Sapir, Mark},
	month = sep,
	year = {2005},
	pages = {959--1058},
}

@incollection{UH,
	title = {Geometry of the complex of curves and of {Teichmüller} space},
	issn = {2523-5133, 2523-5141},
	url = {https://ems.press/books/irma/35/741},
	doi = {10.4171/029},
	urldate = {2026-09-01},
	booktitle = {Handbook of {Teichmüller} {Theory}, {Volume} {I}},
	publisher = {European Mathematical Society},
	author = {Hamenstädt, Ursula},
	month = may,
	year = {2007},
	doi = {10.4171/029},
	pages = {447--467},
}

@article{HPW,
	title = {1-slim triangles and uniform hyperbolicity for arc graphs and curve graphs},
	volume = {17},
	issn = {1435-9855},
	url = {https://ems.press/journals/jems/articles/12078},
	doi = {10.4171/jems/517},
	number = {4},
	urldate = {2026-09-01},
	journal = {Journal of the European Mathematical Society},
	author = {Hensel, Sebastian and Przytycki, Piotr and Webb, Richard C. H.},
	month = apr,
	year = {2015},
	pages = {755--762},
}

@article{MS,
	title = {The geometry of the disk complex},
	volume = {26},
	copyright = {https://www.ams.org/publications/copyright-and-permissions},
	issn = {1088-6834, 0894-0347},
	url = {https://www.ams.org/jams/2013-26-01/S0894-0347-2012-00742-5/},
	doi = {10.1090/S0894-0347-2012-00742-5},
	number = {1},
	urldate = {2026-09-01},
	journal = {Journal of the American Mathematical Society},
	author = {Masur, Howard and Schleimer, Saul},
	month = aug,
	year = {2012},
	pages = {1--62},
}

@article{PSZ,
	title = {Hyperbolic models for {CAT}(0) spaces},
	volume = {450},
	issn = {0001-8708},
	url = {https://www.sciencedirect.com/science/article/pii/S0001870824002573},
	doi = {10.1016/j.aim.2024.109742},
	urldate = {2026-09-01},
	journal = {Advances in Mathematics},
	author = {Petyt, Harry and Spriano, Davide and Zalloum, Abdul},
	month = jul,
	year = {2024},
	pages = {109742},
}

@unpublished{S,
  author  = {Sisto, Alessandro},
  title   = {On metric relative hyperbolicity},
  note    = {arXiv preprint arXiv:1210.8081},
  year    = {2012}
}

@article{V,
	title = {Gromov hyperbolic spaces},
	volume = {23},
	issn = {0723-0869},
	url = {https://www.sciencedirect.com/science/article/pii/S0723086905000113},
	doi = {10.1016/j.exmath.2005.01.010},
	number = {3},
	urldate = {2026-09-01},
	journal = {Expositiones Mathematicae},
	author = {Väisälä, Jussi},
	month = sep,
	year = {2005},
	pages = {187--231},
}
\end{document}